\documentclass[12pt]{amsart}
\usepackage{graphicx,amssymb}
\usepackage{MnSymbol}
\usepackage[all]{xy}
\newtheorem{thm}{Theorem}[section]
\newtheorem{cor}[thm]{Corollary}
\newtheorem{lem}[thm]{Lemma}

\theoremstyle{definition}

\newtheorem{ex}[thm]{Example}

\numberwithin{equation}{section}

\newcommand{\cP}{\mathcal{P}}

\newcommand{\QQ}{\mathbb Q}
\newcommand{\ZZ}{\mathbb Z}

\newcommand{\ra}{\rightarrow}

\DeclareMathOperator{\Ima}{{im}}

\DeclareMathOperator{\Irr}{{Irr}}

\DeclareMathOperator{\End}{{End}}

\DeclareMathOperator{\Ker}{{Ker}}

\def\lra{\longrightarrow}

\def\Q{{\mathbb Q}}

\begin{document}

\title[ ]{Abelian varieties with finite abelian group action}%
\author{Angel Carocca, Herbert Lange and Rub{\'\i} E. Rodr{\'\i}guez}

\address{A. Carocca\\Departamento de Matem\'atica y Estad{\'\i}stica, Universidad de la Frontera, Casilla 54-D, Temuco, Chile}
\email{angel.carocca@ufrontera.cl}

\address{H. Lange\\Department Mathematik,
              Universit\"at Erlangen \\Germany}
\email{lange@math.fau.de}

\address{R. E. Rodr{\'\i}guez\\Departamento de Matem\'atica y Estad{\'\i}stica,
Universidad de la Frontera, Casilla 54-D, Temuco, Chile}
\email{rubi.rodriguez@ufrontera.cl}

\thanks{The  authors were partially supported by grants Fondecyt
1190991, CONICYT PAI Atraccion de Capital Humano Avanzado del Extranjero  PAI80160004 and Anillo ACT 1415 PIA-CONICYT}%

\begin{abstract}
An automorphism of an abelian variety induces a
decomposition of the variety up to isogeny. There are two such
results, namely the isotypical decomposition and Roan's
decomposition theorem. We show that they are essentially the same. Moreover, we generalize in a sense
this result to abelian varieties with action of an arbitrary finite abelian group.
\end{abstract}

\maketitle

\section{Introduction}

Let $A$ be an abelian variety with an action of a finite group $G$.
We always understand by this a faithful action, without further
noticing it. Any such group action induces a decomposition of $A$
as product of $G$-stable abelian subvarieties up to isogeny, called
the isotypical decomposition with respect to $G$ (see \cite[Section
13.6]{lb}). It uses the group algebra $\Q[G]$ and its decomposition
into a product of simple $\Q$-algebras.

In the special case of a cyclic group $G = \langle \alpha \rangle$,
Roan found another isogeny decomposition
(see \cite[Theorem 13.2.8]{lb}), which used only the
analytic representation of $\alpha$ and thus is somewhat simpler to
work out. In the first part of the paper we show that both
decompositions essentially agree (see Theorem \ref{thm3.1}).

It would be useful to have a generalization of Roan's decomposition
for other types of groups. Already in the case of an arbitrary finite abelian group action this would be complicated.
For example it would not suffice to consider the analytic representations of a system of generators of the group, work out decompositions
for every generator and then take intersections.

However we give a method to compute the isotypical decomposition which is slightly weaker (see Theorem \ref{thm}).
In fact, for every irreducible rational representation we work out the corresponding isotypical component of the abelian variety
using intersections of fixed points of certain subgroups of $G$. So there are often a lot of redundant
computations to make, since many isotypical components may be zero. However it is a method for the computation of the decomposition,
which can be done for many groups. We give an example at the end of the paper.

In Section 4 we give some preliminaries for the proof of the theorem and in Section 5 the proof itself.

\section{The decompositions}

\subsection{The isotypical decomposition}
Let $A$ be an abelian variety and let $G$ be a finite group
of automorphisms of $A$.
The action of $G$ on $A$ induces a homomorphism of $\QQ$-
algebras
$$
\rho: \QQ[G] \ra \End_\QQ(A)
$$
of the group algebra of $G$ into the endomorphism algebra
of $A$. As a semisimple algebra, $\QQ[G]$ is a product of
simple $\QQ$-algebras $Q_i$:
$$
\QQ[G] = Q_1 \times \cdots \times Q_r.
$$
This gives a decomposition of $1 \in \Q[G]$,
$$
1 = e_1 + \cdots + e_r
$$
as a sum of central idempotents $e_i$ of $\Q[G]$. The $Q_i$ and thus the
$e_i$ correspond one to one to the irreducible rational
representations $W_i$ of the group $G$.

Now to any idempotent $e$ of $\Q[G]$ one associates an abelian subvariety
$$
A_e := \Ima (e),
$$
where $\Ima (e)$ is defined by the image of $n \rho(e)$ in $A$, where $n$ is any positive integer such that
$n \rho(e) \in \End(A)$. It does not depend on the chosen integer $n$.

For $i = 1, \dots , r$ we denote $A_{e_i}$ also by $A_{W_i}$.
Then the decomposition of 1 of above implies
that the addition map
\begin{equation} \label{e2.1}
A_{W_1} \times \cdots \times A_{W_r} \stackrel{+}{\ra} A
\end{equation}
is an isogeny. This is called the {\it isotypical
decomposition} of $A$ for the action of $G$ (see \cite[Proposition 13.6.1]{lb}). Note that $G$
acts on $A_{W_j}$ via the representation $W_j$ for each $j$ and
the addition map is $G$-equivariant.

\subsection{Roan's decomposition theorem}
Now let $G$ be a cyclic group of order $d$ acting on an
abelian variety $A$, generated by an automorphism $\alpha$.
Suppose
$$
1 \leq d_1 < d_2 < \cdots < d_s
$$
are the orders of the eigenvalues of $\alpha$, meaning the
eigenvalues of the analytic
representation $\rho_a(\alpha)$ of $\alpha$.
Define a filtration of $A$ into $G$-stable abelian
subvarieties
\begin{equation} \label{eq2.2}
0 = Y_{d_s} \subset Y_{d_{s-1}} \subset \cdots \subset Y_{d_1} \subset Y_{d_0} = A
\end{equation}
by
$$
Y_{d_0} := A \quad \mbox{and} \quad Y_{d_i} := \Ima(1_A - \alpha^{d_i})|Y_{d_{i-1}} \quad \mbox{for}\quad i \ge 1,
$$
where $|$ means restricted.

Then denote for $i=1, \dots, s$,
$$
B_{d_i} := \ker ((1_A - \alpha^{d_i})|Y_{d_{i-1}})_0
$$
where the index $0$ stands for the connected component of the kernel
containing 0. In other words, $B_{d_i}$ is the connected
component $\textup{Fix}(\alpha^{d_i}|Y_{d_{i-1}})_0$ containing 0 of the
fixed group of the automorphism $\alpha^{d_i}|Y_{d_{i-1}}$
of $Y_{d_{i-1}}$.  Clearly the $B_{d_i}$ are $G$-stable abelian
subvarieties of $A$ such that
$$
\alpha_i := \alpha|B_{d_i}
$$
is an automorphism of order $d_i$ for all $i$. To be more precise, the eigenvalues of $\alpha_i$ are exactly the eigenvalues of $\alpha$ of order $d_i$.  Then Roan's
decomposition theorem says (see \cite[Theorem 13.2.8]{lb})
that the addition map
\begin{equation} \label{e2.2}
B_{d_1} \times \cdots \times B_{d_s} \stackrel{+}{\lra} A
\end{equation}
is an isogeny. Note that all subvarieties $B_{d_i}$ are of positive
dimension, whereas for the isotypical decomposition this need not be
the case.

\section{The case of a finite cyclic group}

Let $G = \langle \alpha \rangle$ be a cyclic group of order
$d$ acting on an abelian variety $A$ and let the notations be
as in the last section.
The aim of this section is to prove the following theorem.

\begin{thm} \label{thm3.1}
The decompositions \eqref{e2.1} and \eqref{e2.2} are
the same in the following sense:
\begin{enumerate}
\item[(i)] For every
$i = 1, \dots, s$ there is exactly one $j_i \in \{1, \dots , r
\}$ such that
$$
A_{W_{j_i}} = B_{d_i}
$$
and the $j_i$'s are pairwise different.
\item[(ii)]
For all components $A_{W_k}$ with $k \neq j_i$ for $i =1, \dots ,s$ we have
$$ A_{W_k} = 0.
$$
\end{enumerate}
So if one omits these components $A_{W_k}$, the isogenies \eqref{e2.1} and \eqref{e2.2} agree
up to a permutation.
\end{thm}

For the proof of the theorem we need the following lemma.

\begin{lem} \label{lem3.2}
Let the notation be as above. Suppose $\alpha$ admits exactly
$m$ eigenvalues of order $d'$. If $X$ and $Y$ are
$\alpha$-stable abelian subvarieties of dimension $m$ such
that all eigenvalues of the restrictions of $\alpha$ are of
order $d'$,
then
$$
X = Y.
$$
\end{lem}

\begin{proof}
Suppose $X \neq Y$. Let $Z$ denote the abelian subvariety generated by $X$ and $Y$. Then  $Z = X + Y$ is $G$-stable,
of dimension greater than $m$,
and all eigenvalues of $\alpha|Z$ are of order $d'$. This contradicts the fact that $m$ is the exact number of eigenvalues of $\alpha$ of that order.
\end{proof}

\begin{proof}[Proof of the theorem]
By definition of $B_{d_i}$, all eigenvalues of $\alpha|B_{d_i}$ are of
order $d_i$ and these are exactly all eigenvalues of order
$d_i$ of $\alpha$ on $A$ with multiplicities.
According to Lemma \ref{lem3.2} it suffices to show that
exactly one of the subvarieties $A_{W_j}$ has the same
properties.

Recall that the components $A_{W_j}$ correspond to the
rational irreducible representations $W_j$ of $G$; observe that for a cyclic group $G$ of order $d$ the number $r$ of rational irreducible representations of $G$ is equal to the number of divisors of $d$. This may be seen as follows: Let $a_1, \cdots ,a_r$ be all integer divisors of
$d$ in some order. For example, if
$d = p_1^{n_1}\cdots p_q^{n_q}$ is the prime decomposition of
$d$, we can choose
$$
a_1 = 1,\, a_2 = p_1,\, a_3 = p_1^2, \cdots, a_r = d.
$$
Clearly every $d$-th root of unity is a primitive $a_j$-th
root of unity for exactly one $a_j$. Then the rational
irreducible representation $W_j$ of $G$ is given as
(or, more precisely, is complex conjugate to) the direct sum of all characters of order $a_j$.

This implies that the eigenvalues of $\alpha|A_{W_j}$ are exactly all
eigenvalues of order $a_j$ of $\alpha$ on $A$.
So choose $j_i$ such that $d_i = a_{j_i}$. Then Lemma
\ref{lem3.2} implies $B_{d_i} = A_{W_{j_i}}$. The other assertions of the theorem are immediate consequences of this.
\end{proof}

\section{Preliminaries}

\subsection{Idempotents associated to a subgroup}

Let $G$ be a finite abelian group acting on an abelian variety $A$ and
$\rho: \QQ[G] \ra \End_\QQ(A)$ the corresponding homomorphism. To every
subgroup $H$ of $G$ one associates as usual the idempotent $p_H:= \frac{1}{|H|}\sum_{h \in H} h$
and we define the associated abelian subvariety of $A$ by
$A^H := \Ima(p_H).$
It is easy to see that $A^H$ is the maximal abelian subvariety of $A$ on
which $H$ acts trivially.

If $H \subset N \subset G$ are two subgroups of $G$, it follows from
$$
p_N = p_H p_N = p_N p_H
$$
that $A^N \subset A^H$ and that $q:= p_H -p_N$
is an idempotent of $\QQ[G]$. Define
$$
P(A^H/A^N) := \Ima(q).
$$
Since $p_H = p_N + q$, the addition map gives an isogeny
$$
A^N \times P(A^H/A^N) \ra A^H.
$$
Therefore $P(A^H/A^N)$ is called the {\it complementary abelian subvariety} of $A^N$ in $A^H$. Note that it is uniquely determined by the subgroups
$H \subset N$ and in particular is independent of a polarization of $A$.

\subsection{The irreducible rational representations of a finite abelian group}

Now let $G$ be an arbitrary finite abelian group. For any complex irreducible character
$\chi$ of $G$ let
$L_\chi:= \QQ(\chi(g): g \in G)$
denote its field of definition, which is a Galois extension of $\QQ$. For
any $\tau$ in the Galois group $Gal(L_\chi/\QQ)$, the character $\chi^\tau$, defined by
$$
 \chi^\tau(g) := \tau(\chi(g)) \qquad \mbox{for any} \; g \in G,
$$
is an irreducible character of $G$, different from $\chi$ if $\tau$ is not the identity. Then
$$
W:= \bigoplus_{\tau \in Gal(L_\chi/\QQ)} \chi^\tau
$$
is an irreducible rational representation of $G$.
Conversely, every irreducible rational representation arises in this way. We say in this case that each $\chi^\tau$ and $W$ are {\it Galois-associated}.

For any complex irreducible character $\chi$ of $G$ we associate the following subgroup of $G$,
$$
K_\chi:= \Ker(\chi) \subset G.
$$
The following lemma and its corollary are well known. For the convenience of the reader we include the easy
proof.
\begin{lem}
Let $\chi_1$ and $\chi_2$ be two irreducible characters of $G$. Then the
following conditions are equivalent,
\begin{enumerate}
\item $\chi_1$ and $\chi_2$ are Galois-associated to the same irreducible rational representation;
\item $K_{\chi_1} = K_{\chi_2}$.
\end{enumerate}
\end{lem}

\begin{proof}
(1) $\Rightarrow$ (2): By what we have said above, condition (1) means that
$\chi_2 = \tau \circ \chi_1$. Since $\tau$ is an automorphism of
$L_{\chi_1} = L_{\chi_2}$, this implies $K_{\chi_2} = K_{\chi_1}$.

(2) $\Rightarrow$ (1): Note first that, since
$\chi_j$ is a non-trivial homomorphism from $G$ to $S^1 \subset \mathbb{C}^{\ast}$, the
quotient group
$ \; G/K_{\chi_j} \cong \Ima(\chi_j) \; $ is a finite cyclic  subgroup of $S^1$.
For $j = 1, 2$ choose $ x_j \in G$ such that $ \; G/K_{\chi_j} = \langle x_jK_{\chi_j}\rangle \; $.
 Then $ \; \chi_j(x_j) \; $ is a primitive root of unity of order $|G/K_{\chi_j}|$.
So, if
 $ \; W_j \; $ is the rational irreducible representation Galois-associated to $ \; \chi_j, \; $ then $ \; W_j \; $ induces the unique faithful irreducible rational representation of $ \; G/K_{\chi_j}. \; $

 Suppose $ \; K_{\chi_1} = K_{\chi_2}. \; $ Then $ \; W_1 \; $ and $ \; W_2
 \; $ induce the same faithful irreducible rational representations of $ \;
 G/K_{\chi_1}$. But this implies that $ \; W_1 = W_2$ and thus the assertion.
\end{proof}

As an immediate consequence we get,

\begin{cor} \label{c5.2}
With the notation of above there are canonical bijections between the following sets:
\begin{enumerate}
\item classes of Galois-associated complex irreducible characters;
\item irreducible rational representations;
\item subgroups $K$ of $G$ whose quotient is cyclic.
\end{enumerate}
\end{cor}

\section{The general case}

We consider Roan's theorem as a method to compute the isotypical
decomposition of an abelian variety with a $G$-action. For this it suffices to compute the isotypical component $A_W$
for every irreducible rational
representation $W$ of $G$.

So let $A$ be an
abelian variety with an action of an arbitrary finite abelian group
and let $W$ be an irreducible rational representation of $G$. The following theorem gives a method to compute
$A_W$. As we noted in the introduction, this generalizes Roan's theorem in a sense.

If $K \neq G$ denotes the subgroup associated to $W$ according to Corollary \ref{c5.2} consider the subfamily $\cP_K$
of subgroups $H$ of $G$ properly containing  $K$ and minimal with this property, i.e.
$$
\cP_K =\{ H \subset G : K \subset H \textup{ and } [H:K] \textup{ is prime } \}.
$$

Note that the number of elements of the set $\cP_K$ equals the number of prime divisors of $[G:H]$.
Then the isotypical component $A_W$ of $A$ corresponding to $W$ can be computed as follows.

\begin{thm}  \label{thm}
$$
A_W = \left\{ \begin{array}{ll}
              \bigcap_{H \in \cP_K} P(A^{K}/A^H) & if \quad K \neq G,\\
              \\
               A^G & otherwise.
               \end{array}  \right.
$$
\end{thm}

\begin{proof} We may assume that $K \neq G$, the other assertion being trivial. Let $ \; \chi \in \Irr_{\mathbb C}(G) \; $ be
Galois-associated to $W$. So $K = \Ker(\chi)$.

Choose  $ \; x \in G \; $ such that $ \; G/K = \langle xK\rangle = \langle x \rangle K/K \; $ with $ \; x^n \in K, \; $
where $ \;  n = \displaystyle\vert G/K\vert.$
Consider the prime factorization $ \; n = p_1^{\alpha_1}p_2^{\alpha_2}....p_r^{\alpha_r} \; $. Then $x$ decomposes uniquely as
$$ \;
x = x_1x_2...x_r \qquad \mbox{with}  \qquad x_j^{p_j^{\alpha_j}} \in K.
$$
Then $ \; \chi(x_j) \; $ is a $ \; p_j^{\alpha_j}$-th  primitive root of unity, for all $ \; j  = 1, 2, .., r.\; $

Consider a complex irreducible representation $\theta$ of $G$ such that $\langle \rho_{K} , \theta \rangle_G =1$.
Here $\rho_K$ denotes the representation of $G$, induced by the trivial representation of the subgroup $K$. Similarly $\rho_H$
is defined. Moreover, $\langle .,. \rangle_G$ denotes the usual scalar product on the space of complex characters of $G$.

We claim that if $\langle \rho_{H} , \theta \rangle_G =0$ for all $H$ in $\cP_K$, then $\theta$ is
Galois-associated to $\chi$.

To see this, choose $H_j$ in $\cP_K$ such that $[H_j : K] = p_j$. Now by
assumption we have
$$
\langle \theta \: , \: \rho_{K}
 - \rho_{H_j} \rangle_G = 1.
$$
 This implies that $ \; \theta(x_j) \; $  is a $ \; p_j^{\alpha_j}$-th  primitive root of unity. Hence,
if $ \; \langle \theta\: , \:  \rho_{K_W} - \rho_{H} \rangle_G = 1 \; $ for all $H$ in $\cP_K$,  then
$ \; \theta(x_j) \; $  is a $ \; p_j^{\alpha_j}$-th
primitive root of unity for all $ \; j . \; $ Therefore, $ \; K = \ker(\theta) \; $ and $ \; \theta \; $
is Galois-associated  to $ \; \chi. \; $

In this way $ \; W \; $  is the only rational irreducible representation common to all $ \;  \rho_{K} - \rho_{H_j},  \; j = 1,2,...,r \; $,
which implies
$$
A_W = \bigcap_{H \in \cP_K} P(K/H).
$$
\end{proof}

\begin{ex} To give an example for how Theorem \ref{thm} works, consider the group
$$
G = \langle a \rangle \times \langle b \rangle \simeq \ZZ/p^3 \ZZ \times \ZZ/q^2\ZZ
$$
with primes $p$ and $q$, which may be equal or not, the method is the same.
Note that if $p \neq q$, the group $G$ is cyclic, so the method of Section 2.2 may be applied
directly, but Theorem \ref{thm} gives a bit more.

The complex irreducible characters of $G$ are given by
$$
\chi_{(j,k)}(ab) = \omega_{p^3}^j \omega_{q^2}^k \qquad \mbox{for}\qquad 0 \le j \le p^3 -1, \;
0 \le k \le q^2 -1,
$$
where $\omega_{p^3}$ and $\omega_{q^2}$ are primitive $p^3$-rd respectively $q^2$-nd roots of unity.
We consider the character $\chi_{(p^2,q)}$, let $K$ denote its kernel and let $W$ the irreducible rational character Galois-associated to it.
Let $A$ be an abelian variety with an action of $G$. For the computation of the isotypical component $A_W$ we have to distinguish
two cases.

{\bf (a)}: $p \neq q$.
Then the kernel of $\chi_{(p^2,q)}$ is $K = \langle a^pb^q \rangle$, which is properly contained
in $H_1 = \langle ab^q \rangle$ of index $p$ and in $H_2 = \langle a^pb \rangle$ of index $q$
and in no other subgroup of prime index.
So Theorem \ref{thm} gives
$$
A_W = P(A^K/A^{H_1}) \cap P(A^{K}/A^{H_2})
$$
which is a bit more than we get by Roan's method.

Furthermore, $H_1 = \Ker(\chi_{(0,q)})$ and $H_2 = \Ker(\chi_{(p^2,0)})$. So, if $W_1$ and $W_2$ denote the
irreducible rational representations Galois-associated to $\chi_{(0,q)}$ and $\chi_{(p^2,0)}$ respectively,
we obtain
$$
\rho_K - \rho_{H_1} = W \oplus W_2 \quad \mbox{and} \quad \rho_K - \rho_{H_2} = W \oplus W_1.
$$
Hence $G$ acts on $P(A^K/A^{H_1})$ by $W \oplus W_2$ and on $P(A^K/A^{H_2})$ by $W \oplus W_1$.
So $G$ acts on $A_W$ by the representation $W$ as it should.

{\bf (b)}: $p = q$:
In this case the kernel $K$ of $\chi_{(p^2,p)}$ is not cyclic. In fact, $K = \langle ab^{-1} \rangle \times \langle b^p \rangle$, which is of index $p$
in $G$. Hence Theorem \ref{thm} gives
$$
A_W = P(A^{K}/A^G)
$$
and clearly $G$ acts on $A_W$ by the representation $W$.
\end{ex}

{ \bf Acknowledgement:} We would like to thank Gabriele Nebe for pointing out a mistake in the first version
of the paper.

\bibliographystyle{amsplain}

\end{document}